\documentclass[12pt,a4paper,reqno]{article}
\usepackage{amsfonts}
\usepackage{amssymb}
\usepackage{amsthm}
\usepackage{amsmath}
\usepackage{newlfont}
\usepackage{graphicx}
\usepackage{color}
\usepackage{empheq}

\def\r{\mathbb R}

\newtheorem{theorem}{Theorem}[section]

\newtheorem{lemma}[theorem]{Lemma} 

\title{Minimum principles and a priori estimates for some translating soliton type problems}
\author{Rafael L\'opez\\
Departamento de Geometr\'{\i}a y Topolog\'{\i}a\\
 Instituto de Matem\'aticas (IEMath-GR)\\
 Universidad de Granada\\
 18071 Granada, Spain\\
\texttt{rcamino@ugr.es}
\and
Cristian Enache\\
American University of Sharjah\\
 Department of Mathematics and Statistics\\
  University City Road, P.O. Box 26666, Sharjah, UAE\\
\texttt{cenache@aus.edu}
}

\date{}

\begin{document}
\maketitle

 \begin{abstract}
In this paper we are dealing with two classes of mean curvature type problems that generalize the translating soliton problem. A first result proves that the solutions to these problems have unique interior critical points. Using this uniqueness result, we next derive a priori $C^0$ and $C^1$  estimates for the solutions to these problems, by means of some minimum principles for appropriate $P$-functions, in the sense of L.E. Payne.
\end{abstract}

{\it Keywords:} critical points, nodal lines, maximum principles, a priori estimates\\
{\it MSC 2010:} 35B38, 35J93, 53A10, 35J60\\

\section{Introduction and statement of results}\label{se:1}

This paper is devoted to the study of the following two general classes of mean curvature type problems:

  \begin{empheq}[left = \empheqlbrace]{align}
               &\mbox{div}\left( \dfrac{\nabla u}{\sqrt{1+| \nabla u| ^{2}}}
\right) =\left( \dfrac{1}{\sqrt{1+| \nabla u| ^{2}}}
\right) ^{\alpha }\text{ in }\Omega ,\label{eq:1.1} \\
&u=0\ \text{ on }\partial \Omega,\label{eq:1.2}
\end{empheq}

 \begin{empheq}[left = \empheqlbrace]{align}
               &\mbox{div}\left( \dfrac{\nabla v}{\sqrt{1+| \nabla v| ^{2}}}
\right) =\dfrac{1 }{\sqrt{1+| \nabla v| ^{2}}}+\mu
\ \text{ in }\Omega ,  \label{eq:1.3} \\
&v=0 \ \text{ on }\partial \Omega , \label{eq:1.4}
\end{empheq}
where $\Omega \subset \mathbb{R}^{2}$ is a bounded strictly convex domain with
smooth boundary $\partial\Omega$, while $\alpha, \mu>0 $ are some positive constants. We note that if $\kappa=\kappa(s)$ denotes the curvature of $\partial\Omega$ as a planar curve in $\r^2$ computed with respect to the inward orientation, then the strictly convexity of $\Omega$ is equivalent to $\kappa>0$ on $\partial\Omega$.

The motivation for considering these problems has its origin in the singularity theory of the mean curvature flow in $\r^3$ of Huisken and Ilmanen \cite{hs,il,wh}. A {\it translating soliton} is a surface $\Sigma\subset\r^3$ that is  a solution of the mean curvature flow when $\Sigma$ evolves purely by translations along some direction $\vec{a}\in\r^3\setminus\{0\}$. For the initial surface $\Sigma$ in the flow, this implies that $2H=\langle N,\vec{a}\rangle$, where $N$ is a choice of unit normal field. After a change of coordinates, we suppose $\vec{a}=(0,0,1)$. If $\Sigma$ is locally the graph $z=u(x,y)$,   then the identity $2H=\langle N,\vec{a}\rangle$ is now rewritten as
\begin{equation}
\mbox{div}\left( \frac{\nabla u}{\sqrt{1+| \nabla u| ^{2}}}
\right) =  \frac{1}{\sqrt{1+| \nabla u| ^{2}}}.\label{eq:1.5}
\end{equation}
Equation (\ref{eq:1.5}) is called the {\it translating soliton equation} (see Lopez \cite{lo2} for a historical introduction of this equation). Therefore, equations (\ref{eq:1.1}) and (\ref{eq:1.3}) generalize (\ref{eq:1.5}), by taking $\alpha=1$ in (\ref{eq:1.1}) and   $\mu=0$ in (\ref{eq:1.3}), respectively. However, both equations (\ref{eq:1.1})  and (\ref{eq:1.3}) have their own interest. Equation (\ref{eq:1.1}) has been considered in \cite{sh1,sh2,sw} as the extension of the flow of surfaces by powers of the mean curvature $H$. On the other hand, equation (\ref{eq:1.3}) is better understood when we see a solution of (\ref{eq:1.5}) in the  context of manifolds with density (\cite{gr,mo}). More precisely, let $e^\varphi$ be a positive density function in $\r^3$, with $\varphi\in C^\infty(\r^3)$, which serves as a weight for the volume and the surface area. For a given variation $\Sigma_t$ of $\Sigma$, let us denote by $A_\varphi(t)$ and $V_\varphi(t)$ the weighted area and the enclosed weighted volume of $\Sigma_t$, respectively. Then the expressions of the first variation of $A_\varphi(t)$ and $V_\varphi(t)$ are
\begin{equation}
A'_\varphi(0)=-2\int_\Sigma H_\varphi \langle N,\xi\rangle\  dA_\varphi,\quad V_\varphi'(0)=\int_\Sigma \langle N,\xi\rangle\ dA_\varphi, \label{eq:1.6}
\end{equation}
where $\xi$ is the   variational vector field of $\Sigma_t$ and $H_\varphi=H-\langle\nabla\phi,N\rangle/2$ is the so-called weighted mean curvature. If we choose   $\varphi(q)=\langle q,\vec{a}\rangle$, where $\vec{a}=(0,0,1)$, then $H_\varphi=H-\langle N,\vec{a}\rangle/2$. As a consequence of the Lagrange multipliers, we conclude that $\Sigma$ is a critical point of the  area $A_\varphi$ for a  prescribed weighted volume if and only if  $H_\varphi$ is identically constant,  $H_\varphi=\mu/2$: this equation coincides  with (\ref{eq:1.3}) when $\Sigma$ is the graph of $z=u(x,y)$. For $\mu\not=0$, the solutions of equation (\ref{eq:1.3})  invariant by a uniparametric family of rigid motions have been classified by the second author in \cite{lo1}. In a general context, there is a great interest of the solvability of   the mean curvature equation (\ref{eq:1.3}) by replacing the constant $\mu$ by a `forcing term'  $f=f(u,Du)$ (see \cite{ber,jl,ma,ss}).

In this paper we will not study the existence of solutions to problems (\ref{eq:1.1})-(\ref{eq:1.2}) and (\ref{eq:1.3})-(\ref{eq:1.4}). Sufficient conditions on the data for the existence of classical solutions are known in the above bibliography, or more generally, in  the classical article of Serrin \cite{Se69}. Here we are rather interested in obtaining estimates for the solutions to both Dirichlet problems. To this end, we will extensively use the theory of maximum principles developed by L.E. Payne and G.A Philippin in \cite{PP79} for quasilinear elliptic equations of divergence type (see also the book of R. Sperb \cite{Sp81} and the references therein). We will thus develop some new minimum principles for two so-called $P$-functions in the sense of L.E. Payne, that is, for two appropriate functional combinations of the solutions and their derivatives. More precisely, let us consider the following $P$-functions:
\begin{equation}
\Phi( \mathbf{x};\beta) =\frac{2}{\alpha -1}\left(
1+| \nabla u| ^{2}\right)^{\frac{\alpha -1}{2}}-\beta u,\label{eq:1.7}
\end{equation}
\begin{equation}
\Psi ( \mathbf{x};\beta) =\ln \left( \frac{
1+| \nabla v| ^{2}}{\left( 1 +\mu \sqrt{
1+| \nabla v| ^{2}}\right) ^{2}}\right) -\beta v,\label{eq:1.8}
\end{equation}
with $\beta\in\r$, where $u$ and $v$ are solutions to problems (\ref{eq:1.1})-(\ref{eq:1.2}) and (\ref{eq:1.3})-(\ref{eq:1.4}), respectively. Moreover, let us also assume that $\alpha\not=1$, since the case $\alpha=1$, namely,  the  translating soliton equation (\ref{eq:1.5}),  was already investigated by Barbu and Enache \cite{BE13}. The main result of this paper is the following minimum principle.

\begin{theorem}\label{t1} If $\beta \in  [1,2]$, then the auxiliary functions $\Phi ( \mathbf{x};\beta) $  and $\Psi( \mathbf{x};
\beta) $  attain their minimum values on the boundary $\partial \Omega$.
\end{theorem}

As a consequence of these minimum principles for $\Phi ( \mathbf{x};\beta) $ and $\Psi ( \mathbf{x};\beta)$, we derive the following a priori estimates.

\begin{theorem} \label{t2} If $u$ is a solution of problem (\ref{eq:1.1})-(\ref{eq:1.2}), with $\alpha\not=1$, then we have the following lower bound estimates:
\begin{equation}
q_{\min }\geq \left(\kappa_{\max }\right)^{-\frac{2}{\alpha+1 }},  \label{eq:1.9}
\end{equation}
\begin{equation}
-u_{\min }\geq \frac{2}{\alpha -1}\left( \left( \frac{1}{\kappa_{\max }}\right) ^{
\frac{\alpha -1}{\alpha +1}}-1\right) ,   \label{eq:1.10}
\end{equation}
where $q_{\min }=\underset{\partial \Omega }{\min }\left\vert \nabla
u\right\vert $, $u_{\min }=\underset{\overline{\Omega }}{\min\ }u$ and $\kappa_{\max }=\max_{\partial\Omega}\kappa$.
\end{theorem}

\begin{theorem}\label{t3}
If $v$ is a solution of problem (\ref{eq:1.3})-(\ref{eq:1.4}), then we have the following lower bound estimates:
\begin{equation}
q_{\min }\geq \frac{1+\mu }{2\kappa_{\max }},  \label{eq:1.11}
\end{equation}
\begin{equation}
-v_{\min }\geq 2\ln \left( \frac{\left( 1+\mu \right)\sqrt{1+\frac{\left( 1+\mu \right) ^{2}}{%
4\kappa_{\max }^{2}}} }{1+\mu \sqrt{1+\frac{\left( 1+\mu
\right) ^{2}}{4\kappa_{\max }^{2}}}}\right) ,  \label{eq:1.12}
\end{equation}
where $q_{\min }=\underset{\partial \Omega }{\min }\left\vert \nabla
v\right\vert $, $v_{\min }=\underset{\overline{\Omega }}{\min\ }v$ and $\kappa_{\max }=\max_{\partial\Omega}\kappa$.

\end{theorem}

Minimum principles for appropriate $P$-functions, similar to our   results, have been  obtained for several problems of physical or geometrical interests: \cite{E15,E14,M99,M00,PP77, P79,PP12,PS04}. While the corresponding maximum principles are usually easier to obtain as, for instance, in the reference paper   \cite{PP79}, the minimum principles usually require some additional properties of solutions, such as the convexity of the level curves of the solutions or the the uniqueness of their critical points. For the two problems of this paper,  the convexity of level sets of the solutions is still unknown. However, we are able to show that the solutions have an unique critical point, which allows us to adapt a technique employed in \cite{P79} for the case $\alpha =0$ in (\ref{eq:1.1}), and in \cite{BE13} for the case $\mu =0$ in (\ref{eq:1.3}).

The paper is organized as   follows. In Sections 2 and 3 we will give the proof of Theorem 1.1 for problems (\ref{eq:1.1})-(\ref{eq:1.2}) and (\ref{eq:1.3})-(\ref{eq:1.4}), respectively,  together with some preliminary results. In Section 4 we will apply Theorems 1.1 to derive the lower bound estimates from  Theorem \ref{t2}, while in Section 5 we will apply again Theorem \ref{t1}, to obtain the lower bound estimates from  Theorem \ref{t3}. Finally, in Sections 6, we show that some maximum principles developed by Payne and Philippin in \cite{PP79} can be also employed to obtain upper bound estimates which complement the results from Theorems \ref{t2} and \ref{t3} (see Theorems \ref{t4} and \ref{t5}).

\section{Proof of Theorem \ref{t1} for the problem (\ref{eq:1.1})-(\ref{eq:1.2})}\label{se:2}
For the proof of Theorem \ref{t1}, we need first to investigate the number of critical points of the solution to problem (\ref{eq:1.1})-(\ref{eq:1.2}). We point out that  the study of the number critical points  of solutions for elliptic problems is a subject of high interest and the literature is very extensive: here we only refer \cite{sa} in the context of the constant mean curvature equation.

Since the right hand-side of (\ref{eq:1.1}) is positive, the strong  maximum principle implies that $u<0$ in $\Omega$, so $u$ attains its minimum at some interior point of $\Omega$. The next result has its own interest and proves that, in fact, there is only one interior critical point for the solution to problem (\ref{eq:1.1})-(\ref{eq:1.2}).

\begin{theorem}\label{t12} The solution $u$ of problem (\ref{eq:1.1})-(\ref{eq:1.2}) has only one critical point in $\Omega $.
\end{theorem}

\begin{proof} The proof follows the arguments used by Philippin in \cite{P79}, for the case $\alpha =0$ in problem (\ref{eq:1.1})-(\ref{eq:1.2}). For completeness, we also give it briefly here.

Before starting the proof, let us note that, in what follows in this paper, we will always employ the
summation convention over repeated indices (from $1$ to $2$) and adopt
the following notations:
$$
u_{1}=\frac{\partial u}{\partial x_{1}},\text{ }u_{2}=\frac{\partial u}{%
\partial x_{2}},\text{ }u_{ij}=\frac{\partial ^{2}u}{\partial x_{i}\partial
x_{j}},\ \text{ for }i,j\in \left\{ 1,2\right\}.
$$
As for the proof, a first observation is that the solution $u$ of (\ref{eq:1.1})-(\ref{eq:1.2}) is analytic in $\Omega $ (see Nirenberg \cite{Ni53}). We denote $z^{k}=u_{k}$, $k=1,2$, and write equation (\ref{eq:1.1}) in the form
\begin{equation}
\left( 1+| \nabla u| ^{2}\right) \Delta
u-u_{ij}u_{i}u_{j}=\left( 1+| \nabla u| ^{2}\right) ^{
\frac{3-\alpha }{2}}.   \label{eq:2.1}
\end{equation}
Differentiating (\ref{eq:2.1}) with respect to $x_{k}$, $k=1,2$, we see that both $
z^{1}$ and $z^{2}$ satisfy the differential equation
\begin{equation}
\left( \left( 1+| \nabla u| ^{2}\right) \delta
_{ij}-u_{i}u_{j}\right) z_{ij}^{k}+2\left( u_{i}\Delta u-u_{ij}u_{i}-\frac{
3-\alpha }{2}\left( 1+| \nabla u| ^{2}\right) ^{\frac{
1-\alpha }{2}}u_i\right) z_{i}^{k}=0,    \label{eq:2.2}
\end{equation}
in $\Omega $, where $\delta _{ij}$ is the Kronecker symbol. Since equation \eqref{eq:2.2}
is linear in $z$, a linear combination of $z^{1}$ and $z^{2}$ of type
$$z(\theta )=z^{1}\cos \theta +z^{2}\sin \theta,$$
with   $\theta\in\r $, also satisfies equation \eqref{eq:2.2} in $\Omega
$. Therefore, the strong maximum principle implies that $z$
takes its minimum and maximum values on $\partial \Omega $ (\cite[Cor. 3.2]{gt}). On the other
hand, since $u=0$ on $\partial \Omega $, we have
$$z^{k}=\frac{\partial u}{\partial \mathbf{n}}n_{k}\quad\text{ on }\partial \Omega,$$
where $\mathbf{n}=(n_{1},n_{2})$ is the outward unit normal vector on $
\partial \Omega $ and $\partial u/\partial \mathbf{n}$ is the outward normal
derivative of $u$. Then $z(\theta)$ can be now rewritten as
$$z(\theta )=\dfrac{\partial u}{\partial \mathbf{n}}\mathbf{n}\cdot (\cos \theta
,\sin \theta )\quad\text{ on }\partial \Omega.$$
Furthermore, since $u<0$ in $\Omega$, the Hopf boundary point lemma (\cite[Lem. 3.4]{gt}) implies
$$\dfrac{\partial u}{\partial \mathbf{n}}<0\quad\text{ on }\partial \Omega. $$
Let $e^{i\theta}$ be a fixed arbitrary  direction  in the plane $\r^2$. Since $\partial\Omega$ is strictly convex, the normal map $\mathbf{n}:\partial\Omega\rightarrow{\mathbb S}^1$ is one-to-one on the unit circle ${\mathbb S}^1$. We thus deduce that $\mathbf{n}(s)$ is orthogonal to $e^{i\theta}$ at exactly two points and by the definition of $z(\theta)$, the function $z(\theta)$ vanishes along $\partial\Omega$ at exactly   two points.

The proof of Theorem  \ref{t12} is obtained by contradiction. Suppose that there exist at least two critical points of $u$ in $\Omega$, namely,  $P_1$ and $P_2$. Then:

\begin{enumerate}
\item The function $z(\theta)$ is not constant in $\Omega$ because $z(\theta)$ has only two zeros along $\partial\Omega$. Since $z(\theta)$ is analytic,   the critical points of $z(\theta)$ are isolated points.
\item Let $\mathcal{N}_\theta=z(\theta)^{-1}(0)$ be the nodal set of $z(\theta)$. Since $z(\theta)$ is analytic, standard theory asserts that near a critical point of $z(\theta)$, the function $z(\theta)$ is asymptotically approximated by a harmonic homogeneous polynomial. Following Cheng \cite{ch}, $\mathcal{N}_\theta$  is diffeomorphic to the nodal set of the approximating homogeneous polynomial. In particular,     $\mathcal{N}_\theta$ is formed by a set of regular analytic curves at regular points, the so-called nodal lines. On the other hand, in a neighborhood of a critical point, the   nodal lines   form an equiangular system.

We point out that  there is no closed component of $\mathcal{N}_\theta$ contained in $\Omega$. Indeed, if  we assume that $\mathcal{N}_\theta$ encloses a subdomain $\Omega'$ of $\Omega$, then $z(\theta)=0$ along $\partial\Omega'$, so the maximum principle would imply  that $z(\theta)$ is identically $0$ in $\Omega'$, contradicting the fact that $z(\theta)$ is not constant.

\item  We prove that   $\mathcal{N}_\theta$ is formed from only one nodal line. Suppose by contradiction that there exist two nodal lines  $L_1$ and $L_2$. Since $L_1$ an $L_2$  are not closed, then the arcs $L_1$ and $L_2$ end precisely at the two boundary points where $z(\theta)$ vanishes.  Since $\Omega$ is simply-connected, then $L_1$ and $L_2$ enclose at least one subdomain $\Omega'\subset\Omega$: this is impossible by the previous item.

\item As a conclusion, the nodal set $\mathcal{N}_\theta$ is formed from exactly one arc. We now give an orientation to the arc $\mathcal{N}_\theta$ for  each $\theta$: the orientation of $\mathcal{N}_\theta$ is chosen such that we first pass through $P_1$ and then through $P_2$. With respect to this orientation, we are ordering the two boundary points where $z(\theta)$ vanishes. More precisely, let us denote by $P(\theta)$ the initial point of $\mathcal{N}_\theta$, which after passing $P_1$ and then $P_2$, finishes at the other boundary point, which is denoted by $Q(\theta)$.

\item Let us consider $\theta$ varying in the interval $[0,\pi]$.  By the definition of $z(\theta)$,   the functions $z(0)$ and $z(\pi)$ coincides up to the sign, that is, $z(0)=-z(\pi)$ and thus the nodal lines $\mathcal{N}_0$ and $\mathcal{N}_\pi$ coincide as sets of points. However,   when    $\theta$ runs in $[0,\pi]$, the ends points of $\mathcal{N}_0$ interchange their position when $\theta$ reaches the value $\theta=\pi$, leading to the nodal line $\mathcal{N}_\pi$. Therefore, according to the chosen orientation in $\mathcal{N}_\theta$,  $P(0)=Q(\pi)$ and $P(\pi)=Q(0)$. Since all the arcs $\mathcal{N}_\theta$ pass first through $P_1$ and then through $P_2$, this interchange of the end points between $\mathcal{N}_0$ and $\mathcal{N}_\pi$ would imply  the existence of another nodal line for  $z(\pi)$. But this is impossible, by item 3. This contradiction    completes the proof of Theorem \ref{t12}.
\end{enumerate}
\end{proof}

Now, once the uniqueness of the interior critical point of $u$ is proved, using some rotation and/or translation if necessary, we can choose the coordinates axes such that the unique critical point of $u$ is located at $\mathbf{O}$, the origin of the coordinate system. Then $\mathbf{O}$ is the unique point of global minimum for $u$, so   we   have $u_{11}(\mathbf{O})\geq 0$ and $u_{22}(\mathbf{O})\geq 0$. The next lemma shows that in fact these inequalities are strict.

\begin{lemma}\label{l22}
  If $u$ is the solution of problem \eqref{eq:1.1}-\eqref{eq:1.2}, then
  $$u_{11}(\mathbf{O})>0,\text{\qquad }u_{22}(\mathbf{O})>0. $$
\end{lemma}
 \begin{proof}
The proof is obtained by contradiction. Suppose that $u_{11}(\mathbf{O})=0$ (a similar argument will work if we assume instead that $u_{22}(\mathbf{O})=0$). If the function $z^1=u_1$ is constant in $\Omega$, then $u$ depends only on the variable $x_2$ and the boundary condition (\ref{eq:1.2}) is impossible. Thus $z^1$ is a non constant analytic function. Since $z^1$ vanishes at $\mathbf{O}$ as well as $z^1_1$ and $z^1_2$, then the function $z^1$ vanishes at $\mathbf{O}$ with a finite order $m\geq 1$. Consequently there exist at least two nodal lines of  $z^1$ which form an equiangular system in a neighborhood of $\mathbf{O}$. However we have already proved in Theorem \ref{t12} the existence of exactly one nodal line, unless $z^1$ is constant in $\Omega$, so that we achieve a contradiction.
\end{proof}

\begin{lemma}\label{l23}
 If $\beta \in \left[ 1,2\right] $, then the auxiliary function $\Phi( \mathbf{x};\beta) $ attains its minimum value at the critical point of $
u$ or on the boundary $\partial \Omega $.
\end{lemma}
\begin{proof}
Differentiating successively (\ref{eq:1.7}), we have
\begin{equation}
\Phi _{k}=2\left( 1+| \nabla u| ^{2}\right) ^{\frac{
\alpha -3}{2}}u_{ik}u_{i}-\beta u_{k},   \label{eq:2.8}
\end{equation}
respectively
\begin{equation}
\Phi _{kl}=2\left( \alpha -3\right) \left( 1+| \nabla u|
^{2}\right) ^{\frac{\alpha -5}{2}}u_{ik}u_{i}u_{jl}u_{j}+2\left(
1+| \nabla u| ^{2}\right) ^{\frac{\alpha -3}{2}}\left(
u_{ikl}u_{i}+u_{il}u_{il}\right) -\beta u_{kl}.   \label{eq:2.9}
\end{equation}
We now remind the following identity
\begin{equation}
u_{ik}u_{ik}| \nabla u| ^{2}=| \nabla
u| ^{2}\left( \Delta u\right) ^{2}+2u_{ij}u_{i}u_{kj}u_{k}-2\left(
\Delta u\right) u_{ij}u_{i}u_{j},   \label{eq:2.10}
\end{equation}
which holds only in $\mathbb{R}^{2}$ (see \cite{PS04}). Making use of (\ref{eq:2.10}), after some manipulations (see \cite[Eq. (2.15)]{PP79}), we obtain
\begin{equation}
\Delta \Phi -\frac{1}{1+| \nabla u| ^{2}}\Phi
_{kj}u_{k}u_{j}+W_{k}\Phi _{k}=\left( \beta -2\right) \left( 1+|
\nabla u| ^{2}\right) ^{-\frac{\alpha +1}{2}}\left( \beta -1+\beta
\frac{\alpha }{2}| \nabla u| ^{2}\right) ,
\label{eq:2.11}
\end{equation}
where $W_{k}$ is a smooth vector function which is singular at the critical
point of $u$. We observe that the right hand-side of (\ref{eq:2.11}) is non-positive, because $\beta-2\leq 0$, and the other two parentheses are positive. Therefore, the conclusion of Lemma \ref{l23} follows now from (\ref{eq:2.11}), as a direct consequence of the strong maximum principle.
\end{proof}

\begin{lemma} \label{l24}
If $\beta \in [1,2]$, then the auxiliary function $\Phi ( \mathbf{
x};\beta) $ cannot be identically constant on $\overline{\Omega }$.
\end{lemma}

\begin{proof}
If $\beta \in [1,2)$, then obviously no constant $\Phi ( \mathbf{
x};\beta) $ can satisfy (\ref{eq:2.11}) because the right hand-side is positive. Therefore, it remains to investigate the case  $\beta =2$. In such a case, we assume contrariwise that $\Phi(\mathbf{x};2)$ is constant on $\overline{\Omega }$. By the definition of $\Phi(\mathbf{x};2)$ and the fact that  $u=0$ on $\partial \Omega $, we deduce that
$|\nabla u|$ is constant on $\partial \Omega $. Therefore,
according to a  symmetry result of Serrin (\cite{Se71}), the domain $\Omega $
must be a disk and the solution to problem (\ref{eq:1.1})-(\ref{eq:1.2}) must be radial, that is,  $
u=u\left( r\right) $, with $r=| \mathbf{x}| $. Now, in radial coordinates, equation (\ref{eq:1.1}) can be rewritten as
\begin{equation}
u_{rr}+\frac{1}{r}u_{r}( 1+u_{r}^{2}) =(1+u_{r}^{2})
^{\frac{3-\alpha }{2}}.  \label{eq:2.12}
\end{equation}
On the other hand, since $\Phi ( \mathbf{x};2)$ is constant, we
have that $\partial \Phi /\partial r=\Phi _{,k}u_{,k}=0$, so that
$u_{rr}=(1+u_{r}^{2}) ^{\frac{3-\alpha }{2}}$ and (\ref{eq:2.12}) becomes
$$\frac{1}{r}u_{r}(1+u_{r}^{2}) =0, $$
which is impossible, since $u_{r}\neq 0$ for $r\neq 0$. We have thus obtained a contradiction and the proof is achieved.
\end{proof}

We are now in position to prove Theorem \ref{t1}. The proof is obtained by contradiction. Let us assume that the minimum of $\Phi ( \mathbf{x};\beta) $ occurs at the   critical point $\mathbf{O}$
of $u$. We distinguish two cases.

\begin{enumerate}
\item Case $\beta \in (1,2]$. Using the fact that $u_{1}( \mathbf{O}) =u_{2}\left(
\mathbf{O}\right) =u_{12}( \mathbf{O}) =0$, we evaluate (\ref{eq:2.8}) and
(\ref{eq:2.9}) at $\mathbf{O}$ to obtain
$$\Phi _{1}(\mathbf{O};\beta )=\Phi _{2}(\mathbf{O};\beta)=0,$$
respectively
$$
\begin{array}{l}
\Phi _{11}(\mathbf{O};\beta )=2u_{11}^{2}(\mathbf{O})-\beta u_{11}(\mathbf{O}%
),  \\
\Phi _{12}(\mathbf{O};\beta )=0,  \\
\Phi _{22}(\mathbf{O};\beta )=2u_{22}^{2}(\mathbf{O})-\beta u_{22}(\mathbf{O}%
).
\end{array}$$
Since $\Phi(\mathbf{x};\beta )$ attains its minimum in $\mathbf{O}$, we have
\begin{equation}
\begin{split}
0\leq \Phi_{11}(\mathbf{O};\beta )& =u_{11}(\mathbf{O})(2u_{11}(\mathbf{O})-\beta
),   \\
0\leq \Phi_{22}(\mathbf{O};\beta )& =u_{22}(\mathbf{O})(2u_{22}(\mathbf{O})-\beta
).
\end{split}
 \label{eq:2.16}
\end{equation}
It follows then from Lemma \ref{l22} and (\ref{eq:2.16}) that
\begin{equation}
2u_{11}(\mathbf{O})-\beta \geq 0,\qquad 2u_{22}(\mathbf{O})-\beta \geq 0. \label{eq:2.17}
\end{equation}
Summing now these two last inequalities, if follows that
\begin{equation}
\Delta u(\mathbf{O})-\beta \geq 0.    \label{eq:2.18}
\end{equation}
On the other hand, evaluating \eqref{eq:1.1} at $\mathbf{O}$, we get
\begin{equation}
\Delta u(\mathbf{O})=1.   \label{eq:2.19}
\end{equation}
Inserting now (\ref{eq:2.19}) into (\ref{eq:2.18}), we conclude that $\beta \leq 1$,  which contradicts the assumption that $\beta >1$, so that Theorem 1.1 is
proved in this case.

\item Case $\beta =1$. We repeat a continuity argument used by Philippin and Safoui in \cite{PS04}. From Lemma \ref{l24} we know that for
all $\beta \in  [1,2]$ the auxiliary functions $\Phi( \mathbf{x};\beta) $ take its minimum value either on the boundary $\partial
\Omega $ or at the critical point of $u$. On the
other hand, from the previous case $\beta<1$, we also know that  $\Phi
( \mathbf{x};\beta) $ takes its minimum value on $\partial
\Omega $ for all $\beta \in \left( 1,2\right] $. However, when $\beta $
decreases continuously from $2$ to $1$, the points at which $\Phi (\mathbf{x};\beta) $ takes its minimum value have to move continuously. Therefore, they cannot jump away from $\partial \Omega $ at the interior critical point of $u$. This contradiction thus proves  Theorem \ref{t1}, when $\beta=1$.
\end{enumerate}

\section{Proof of Theorem \ref{t1} for the
 problem (\ref{eq:1.3})-(\ref{eq:1.4})}\label{se:3}

The  proof of Theorem \ref{t1} for problem (\ref{eq:1.3})-(\ref{eq:1.4}) is similar to the one given in the previous section for problem (\ref{eq:1.1})-(\ref{eq:1.2}). We will thus follow the same steps and only  present  the  differences.  Recall that the constant $\mu$ in (\ref{eq:1.3}) is positive.

Again, let us   notice that, since the right hand-side of equation (\ref{eq:1.3}) is positive, the strong maximum principle implies that the solution $v$ satisfies $v<0$ in $\Omega$. Next, we have to show that $v$  has only one critical point in $\Omega$. To this end, the argument is the same as the one from the previous section. More precisely, we first differentiate equation (\ref{eq:1.3}) with respect to $x_k$, to obtain
$$
\left( \left( 1+| \nabla v| ^{2}\right) \delta
_{ij}-v_{i}v_{j}\right) z_{ij}^{k}+2\left( v_{i}\Delta v-v_{ij}u_{i}-v_i-\frac32\mu(1+|\nabla v|^2)\right) z_{i}^{k}=0. $$
Therefore, the strong maximum principle can be applied to this equation and the remaining part of the proof is identical to what one has already seen in the proof of Theorem \ref{t12}. This means that $v$ has a unique critical point, which is a point of global minimum. As before, we may assume that this point is   the origin $\mathbf{O}$ and clearly a result similar to Lemma \ref{l22} can be easily derived in this case, to obtain that $v_{11}(\mathbf{O})>0$ and $v_{22}(\mathbf{O})>0$.

We now prove a result which analogous to Lemma \ref{l23}.

\begin{lemma}\label{l33}
If $\beta \in \left[ 1,2\right] $, then the auxiliary function $\Psi \left(
\mathbf{x};\beta \right) $ attains its minimum value at the critical point
of $v$ or on the boundary $\partial \Omega $.
\end{lemma}

\begin{proof}
Differentiating successively (\ref{eq:1.8}), we obtain
\begin{equation}
\Psi _{k}=\frac{2v_{ik}v_{i}}{\left( 1+| \nabla v|
^{2}\right) \left( 1 +\mu \sqrt{1+| \nabla v| ^{2}}
\right) }-\beta v_{k},   \label{eq:3.2}
\end{equation}
respectively
\begin{equation}
\Psi _{kl}=\frac{2v_{ikl}v_{i}+2v_{ik}v_{il}}{
1+| \nabla v| ^{2} +\mu \left( 1+| \nabla
v| ^{2}\right) ^{3/2} }-\frac{4\left(1 +\frac{3}{2}
\mu \sqrt{1+| \nabla v| ^{2}}\right)
v_{ik}v_{i}v_{jl}v_{j}}{\left(  1+| \nabla
v| ^{2}  +\mu ( 1+| \nabla v|
^{2})^{3/2}\right)^{2}}-\beta v_{kl}.   \label{eq:3.3}
\end{equation}
Then the equation corresponding here to (\ref{eq:2.11}) is
\begin{equation}
\Delta \Psi -\frac{1}{1+| \nabla v| ^{2}}\Psi
_{kj}v_{k}v_{j}+W_{k}\Psi _{k} =\frac{\beta -2}{1+| \nabla
v| ^{2}}\left\{ \left( \beta -1\right) \left(1 +\mu \sqrt{
1+| \nabla v| ^{2}}\right) +  \frac{\beta }{2}
| \nabla v| ^{2}\right\}. \label{eq:3.4}
\end{equation}
If $\beta \in \left[ 1,2\right]$, then $\beta-2\leq 0$ and the bracket in the above identity is positive (here we have used the fact that $\mu>0$). The result follows now as a direct consequence of the strong maximum principle.
\end{proof}

\begin{lemma}
If $\beta \in  [1,2]$, then the auxiliary function $\Psi \left(
\mathbf{x};\beta \right) $ cannot be identically constant on $\overline{
\Omega }$.
\end{lemma}

\begin{proof} If $\beta\in  [1,2)$, then clearly no constant $\Psi \left(
\mathbf{x};\beta \right)$ can satisfy (\ref{eq:3.4}). Therefore, it remains to investigate the case $\beta=2$. In such a case, if we assume that $\Psi \left(
\mathbf{x};2 \right)$ is constant on $\overline{
\Omega }$, we may obtain again that $v$ is a radial function, that is $v=v(r)$, and $\Omega$ is a disk. In radial coordinates,  equation (\ref{eq:1.3}) can be rewritten as
\begin{equation}
v_{rr}+\frac{1}{r}v_r(1+v_r^2)=1+v_r^2+\mu(1+v_r^2)^{3/2}. \label{eq:3.5}
\end{equation}
On the other hand, since $\Psi ( \mathbf{x};2)$ is constant, we
have that $\partial \Psi /\partial r=\Psi _{,k}u_{,k}=0$, so that
$u_{rr}=(1+u_{r}^{2})+\mu (1+u_{r}^{2})^{3/2}$ and (\ref{eq:3.5}) thus becomes
$$\frac{1}{r}u_{r}(1+u_{r}^{2}) =0,$$
which is impossible, since $u_{r}\neq 0$ for $r\neq 0$.
\end{proof}

We are now ready to prove Theorem \ref{t1}. The proof is obtained again by contradiction. Suppose that  the minimum of $\Psi ( \mathbf{x};\beta) $ is attained at the critical point $\mathbf{O}$ of $v$. We distinguish two cases.

\begin{enumerate}
\item Case $\beta \in (1,2]$. Using the fact that $v_{1}( \mathbf{O}) =v_{2}\left(
\mathbf{O}\right) =v_{12}( \mathbf{O}) =0$, we evaluate (\ref{eq:3.2}) and
(\ref{eq:3.3}) at the origin to find
$$\Psi _{1}(\mathbf{O};\beta )=\Psi _{2}(\mathbf{O};\beta )=0,$$
respectively
$$\begin{array}{l}
\Psi _{11}(\mathbf{O};\beta )=2\dfrac{v_{11}^{2}( \mathbf{O}) }{
1 +\mu }-\beta v_{11}( \mathbf{O}) ,   \\
\Psi _{12}( \mathbf{O};\beta) =0,   \\
\Psi _{22}( \mathbf{O};\beta) =2\dfrac{v_{22}^{2}}{1 +\mu }
( \mathbf{O}) -\beta v_{22}( \mathbf{O}) .
\end{array}$$
Since $\Psi (\mathbf{x};\beta )$ attains its minimum in $\mathbf{O}$, we
have \
\begin{equation}
\begin{split}
0\leq \Psi _{11}(\mathbf{O};\beta )& =v_{11}(\mathbf{O})\left(2\frac{v_{11}(\mathbf{O})
}{1 +\mu }-\beta \right),   \\
0\leq \Psi _{22}(\mathbf{O};\beta )& =v_{22}(\mathbf{O})\left(2\frac{v_{22}(\mathbf{O})
}{1 +\mu }-\beta \right)\geq 0.
\end{split}
 \label{eq:3.9}
\end{equation}
Using now   that $v_{ii}(\mathbf{O})>0$, for $i=1,2$, inequalities (\ref{eq:3.9}) imply
$$2\frac{v_{11}(\mathbf{O})}{1 +\mu }-\beta \geq 0,\qquad 2\frac{
v_{22}(\mathbf{O})}{1 +\mu }-\beta \geq 0.$$
Summing now these two last inequalities, we obtain
\begin{equation}
2\Delta v(\mathbf{O})\geq \left( 1 +\mu \right) \beta .
 \label{eq:3.11}
\end{equation}
On the other hand, evaluating equation \eqref{eq:1.3} at  $\mathbf{O}$, we find
\begin{equation}
\Delta v(\mathbf{O})=1 +\mu .   \label{eq:3.12}
\end{equation}
Inserting now (\ref{eq:3.12}) into (\ref{eq:3.11}) and using the fact that $1+\mu>0$, we conclude that $\beta \leq 1$,  which contradicts the assumption that $\beta >1$, so that the proof of Theorem \ref{t1} is achieved in this case.

\item Case $\beta =1$. The same continuity argument employed in the case of problem (\ref{eq:1.1})-(\ref{eq:1.2})  in the previous section  can be repeated here to show that  Theorem \ref{t1} also holds in this case.
\end{enumerate}

\section{Proof of Theorem 1.2}\label{se:4}

From Theorem 1.1 we know that $\Phi ( \mathbf{x};\beta) $ takes its
minimum value at some point $\mathbf{Q}_\beta\in \partial \Omega $. Here we emphasize the dependence of the points $\mathbf{Q}_\beta$ on the parameter $\beta$. This
implies that
\begin{equation}
\frac{2}{\alpha -1}\left( 1+| \nabla u| ^{2}\right) ^{
\frac{\alpha -1}{2}}-\beta u\geq \frac{2}{\alpha -1}\left( 1+q_{m}^{2}\right) ^{
\frac{\alpha -1}{2}},   \label{eq:4.1}
\end{equation}
where we remind that $q_{m}$ is the minimum value of $| \nabla u| $ on $
\partial \Omega $. Evaluating (\ref{eq:4.1}) at the unique minimal point of $u$, we find
$$-\beta u_{\min }\geq \frac{2}{\alpha -1}\left(\left( 1+q_{m}^{2}\right) ^{\frac{
\alpha -1}{2}}-1\right). $$
The left hand-side of this inequality is positive and attains its minimum when $\beta=1$, hence
\begin{equation}
- u_{\min }\geq \frac{2}{\alpha -1}\left(\left( 1+q_{m}^{2}\right) ^{\frac{
\alpha -1}{2}}-1\right).  \label{eq:4.3}
\end{equation}
 Next, we construct a lower bound for $q_{m}$ in terms of the curvature $\kappa(s)$ of $\partial
\Omega $.  Let $\mathbf{Q}=\mathbf{Q}_1$. Since $\Phi( \mathbf{x};1) $ takes its minimum value at $\mathbf{Q}$, we have $\partial \Phi ( \mathbf{x};1) /\partial \mathbf{n}\leq 0
$ at $\mathbf{Q}$, or equivalently,
\begin{equation}\label{otro}
2\left( 1+u_{n}^{2}\right) ^{\frac{\alpha -3}{2}}u_{n}u_{nn}-u_{n}\leq 0\
\text{ at }\mathbf{Q},
\end{equation}
where $u_{n}$ and $u_{nn}$ are the first and second outward normal
derivatives of $u$ on $\partial \Omega $. As $u<0$ in $\Omega$ and $u=0$ along $\partial\Omega$, then $u_n>0$ and $u_n=|\nabla u|$ on $\partial\Omega$. Thus the above inequality (\ref{otro}) becomes
\begin{equation}
2\left( 1+u_{n}^{2}\right) ^{\frac{\alpha -3}{2}} u_{nn} \leq 1\
\text{ at }\mathbf{Q}.   \label{eq:4.5}
\end{equation}
Now, since the boundary $\partial \Omega $ is smooth,  equation (\ref{eq:1.1})
   can be rewritten in normal coordinates along $\partial\Omega$ as
$$\frac{u_{nn}}{\left( 1+u_{n}^{2}\right) ^{3/2}}+\frac{\kappa u_{n}}{\left(
1+u_{n}^{2}\right) ^{1/2}}=\left( 1+u_{n}^{2}\right) ^{-\frac{\alpha }{2}}
\text{ on }\partial \Omega ,$$
or equivalently,
\begin{equation}
 u_{nn} + \kappa u_{n} (
1+u_{n}^{2} ) = ( 1+u_{n}^{2} ) ^{\frac{3-\alpha }{2}}
\text{ on }\partial \Omega.  \label{eq:4.7}
\end{equation}
Inserting   (\ref{eq:4.7}) into (\ref{eq:4.5}), we obtain
$$1\leq 2\kappa(\mathbf{Q})q_m(1+q_m^2)^{\frac{\alpha-1}{2}}\leq 2\kappa(\mathbf{Q})\left( 1+q_{m}^{2}\right) ^{\frac{\alpha +1}{2}}$$
where for the last inequality  we have used  that $2q_m\leq 1+q_m^2$. It then follows that
$$\frac{1}{\kappa(\mathbf{Q})} \leq
\left( 1+q_{m}^{2}\right) ^{\frac{\alpha +1}{2}}. $$
Hence
\begin{equation}
\left( 1+q_{m}^{2}\right) ^{\frac{\alpha -1}{2}}\geq \left( \frac{1}{\kappa(\mathbf{Q})}\right) ^{\frac{\alpha -1}{\alpha +1}}\geq  \left( \frac{1}{\kappa_{\max
}}\right) ^{\frac{\alpha -1}{\alpha +1}}, \label{eq:4.10}
\end{equation}
from which inequality (\ref{eq:1.9}) follows. Moreover, inserting  (\ref{eq:4.10}) into (\ref{eq:4.3}), we also obtain the inequality (\ref{eq:1.10}) and the proof of Theorem \ref{t2} is thus achieved.

\section{Proof of Theorem \ref{t3}}\label{se:5}

The idea of proof is similar to the one already employed in the previous section to prove  Theorem \ref{t2}. From Theorem \ref{t1} we know that $\Psi ( \mathbf{x};1) $ takes its
minimum value at some point $\mathbf{Q}\in \partial \Omega $. This
implies that
\begin{equation}
 \ln \left( \frac{1+| \nabla v| ^{2}}{
\left(1 +\mu \sqrt{1+| \nabla v| ^{2}}\right) ^{2}}
\right) -v\geq \ln \left( \frac{1+q_{m}^{2}}{\left(
1 +\mu \sqrt{1+q_{m}^{2}}\right) ^{2}}\right).  \label{eq:5.1}
\end{equation}
 Evaluating now (\ref{eq:5.1}) at the unique minimal point of $v$, we obtain
\begin{equation}
-v_{\min }\geq  \ln \left( \frac{\left( 1+q_{m}^{2}\right)
 ( 1 +\mu  )^2 }{\left(1 +\mu \sqrt{1+q_{m}^{2}}\right)
^{2}}\right) . \label{eq:5.2}
\end{equation}
From  the facts that $\partial \Psi ( \mathbf{x};1) /\partial \mathbf{n}\leq 0
$ and $v_n>0$ on $\partial\Omega$, it follows that
\begin{equation}
\frac{2 v_{nn}}{\left( 1+v_{n}^{2}\right) (1 +\mu \sqrt{
1+v_{n}^{2}}) } \leq 1\ \text{ at }\mathbf{Q}.
 \label{eq:5.3}
\end{equation}
On the other hand, equation (\ref{eq:1.3}) can be rewritten in normal coordinates along $\partial\Omega$, as
$$\frac{v_{nn}}{\left( 1+v_{n}^{2}\right) ^{3/2}}+\frac{\kappa v_{n}}{\left(
1+v_{n}^{2}\right) ^{1/2}}=\frac{1 }{\sqrt{1+v_{n}^{2}}}+\mu, $$
or, equivalently,
\begin{equation}
v_{nn}=(1+v_n^2)(1-\kappa v_n)+\mu(1+v_n^2)^{3/2}, \label{eq:5.5}
\end{equation}
where $\kappa(s)$ is the curvature of $\partial \Omega$. Inserting now the value of $v_{nn}$ from (\ref{eq:5.5}) into (\ref{eq:5.3}),  we obtain after some simplifications
$$\frac{1}{\kappa(\mathbf{Q})}\leq \frac{2q_{m}}{1 +\mu \sqrt{1+q_{m}^{2}}}\leq  \frac{2q_{m}}{1 +\mu },  $$
from which inequality (\ref{eq:1.11}) follows. Finally, using (\ref{eq:1.11}) and the fact that the function $f(x)=x/(1+\mu x)$ is increasing, we easily deduce from (\ref{eq:5.2}) the desired inequality (\ref{eq:1.12}).

\section{Some final remarks}\label{se:6}

For both problems (\ref{eq:1.1})-(\ref{eq:1.2}) and (\ref{eq:1.3})-(\ref{eq:1.4}), some maximum principles for $P$-functions have been already obtained by Payne and Philippin in \cite{PP79}. We can use them in what follows, to derive some upper bound estimates which complement the bounds given in Theorem \ref{t2} and \ref{t3}.

\subsection{An upper bound for $-u_{\min}$}

From \cite[Cor. 1]{PP79}, we know that the function $\Phi (\mathbf{x};2)$ takes its maximum value at the
(only) critical point of $u$. This implies
$$\frac{1}{\alpha -1}\left( 1+|\nabla u|^{2}\right) ^{\frac{\alpha -1}{2}}-%
\frac{1}{\alpha -1}\leq u-u_{\min }.   $$
Therefore, if $\alpha >1$, this inequality   leads to
\begin{equation}
|\nabla u|^{2}\leq \left( (\alpha -1)(u-u_{\min })+1\right) ^{\frac{2}{%
\alpha -1}}-1.  \label{eq:6.2}
\end{equation}
Next, we use inequality (\ref{eq:6.2}) to derive an upper bound for $-u_{\min }$. Let $%
P$ be a point where $u=u_{\min }$ and $Q$ be a point on $\partial \Omega $
nearest to $P$.  Let $r$ measure the distance from $P$ to $Q$ along the ray
connecting $P$ and $Q$. Clearly we have%
\begin{equation}
\frac{du}{dr}\leq \left\vert \nabla u\right\vert .  \label{eq:6.3}
\end{equation}%
Integrating now (\ref{eq:6.3}) from $P$ to $Q$ along the ray connecting $P$ to $Q$, and making use of (\ref{eq:6.2}), we obtain  %
\begin{equation}
I:= \int_{u_{\min }}^0  \frac{du}{\sqrt{\left( 1+\left(
\alpha -1\right) \left( u-u_{\min }\right) \right) ^{\frac{2}{\alpha -1}}-1}}%
\leq  \int_P^Qdr=\left\vert PQ\right\vert \leq d,
 \label{eq:6.4}
\end{equation}%
where $d$ is the radius of the largest ball inscribed in $\Omega $. Next,
using the substitution $v=\left( 1+\left( \alpha -1\right) \left( u-u_{\min
}\right) \right)^{-\frac{1}{\alpha -1}}$, we have%
\begin{eqnarray}
I &= & \int_{1}^{\left( 1-\left( \alpha -1\right) u_{\min }\right)^{-%
\frac{1}{\alpha -1}}}\frac{v^{1-\alpha }}{\sqrt{1-v^{2}}}dv  \geq  \int_{1}^{\left( 1-\left( \alpha -1\right) u_{\min }\right)^{-%
\frac{1}{\alpha -1}}}\frac{dv}{\sqrt{1-v^{2}}}  \label{6.51}\\
& = & \cos ^{-1}\left(\frac{1}{\left( 1-\left( \alpha -1\right) u_{\min }\right)
^{\frac{1}{\alpha -1}}}\right), \label{6.5}
\end{eqnarray}

where we have used the fact that $v\leq 1$ when $\alpha >1$,  to derive the
inequality in (\ref{6.51}). From (\ref{eq:6.4}) and (\ref{6.5}) we obtain
$$\cos ^{-1}\left(\frac{1}{\left( 1-\left( \alpha -1\right) u_{\min }\right)
^{\frac{1}{\alpha -1}}}\right)\leq d.$$
In order to solve this inequality for $u_{\min}$, we must require that $d<\pi/2$. Hence we conclude the following estimate.

\begin{theorem}\label{t4} Let $d$ be the radius of the largest ball inscribed in $\Omega $.
If $\alpha >1$ and $d<\pi/2$, then the solution $u$ to problem (\ref%
{eq:1.1})-(\ref{eq:1.2}) satisfies the inequality
$$-u_{\min }\leq \frac{1}{\alpha -1}\left( \left( \frac{1}{\cos (d)}\right)
^{\alpha -1}-1\right).$$
\end{theorem}

\subsection{An upper bound for $-v_{\min}$}

From \cite[Cor. 1]{PP79}, the function $\Psi (\mathbf{x};2)$ takes its maximum value at the
(only) critical point of $v$. This implies that
$$
\ln \left( \frac{1+\left\vert \nabla v\right\vert ^{2}}{\left( 1+\mu \sqrt{%
1+\left\vert \nabla v\right\vert ^{2}}\right) ^{2}}\right) -\ln \left( \frac{%
1}{\left( 1+\mu \right) ^{2}}\right) \leq 2v-2v_{\min } ,  $$
so that, after some manipulations, we find
\begin{equation}
\frac{\left( 1+\mu \right) ^{2}\left( 1+\left\vert \nabla v\right\vert
^{2}\right) }{\left( 1+\mu \sqrt{1+\left\vert \nabla v\right\vert ^{2}}%
\right) }\leq e^{2v-2v_{\min }}.   \label{eq:6.9}
\end{equation}

Since $\mu>0$, the left hand-side of (\ref{eq:6.9}) is obviously larger than $1+\left\vert
\nabla v\right\vert ^{2}$, so from (\ref{eq:6.9}) we are lead to the following
inequality%
\begin{equation}
\left\vert \nabla v\right\vert ^{2}\leq e^{2\left( v-v_{\min }\right) }.
  \label{eq:6.10}
\end{equation}%
Next, following the steps of the previous subsection, from (\ref{eq:6.10}), which represents exactly inequality (\ref{eq:6.9}) for $%
\mu =0$, one may obtain the
following result (see also \cite{PP79}, where the case $%
\mu =0$ was already investigated):

\begin{theorem}\label{t5} Let $d$ be the radius of the largest ball inscribed in $\Omega $.
If $d<\pi/2$, then the solution $v$  to problem (\ref{eq:1.3})-(%
\ref{eq:1.4}) satisfies the following inequality%
\begin{equation}
-v_{\min }\leq \ln \left( \frac{1}{\cos \left( d\right) }\right) .
 \label{eq:6.11}
\end{equation}%
\end{theorem}

\subsection{Extensions and optimality}

We conclude this paper with the following two remarks about the extensions to higher dimensions and the optimality of the bounds found in this paper.

1. The most important ingredient in the proof of Theorem 1.1 is the result about the uniqueness of the critical point of solutions. An extension of our idea of proof to higher dimension doesn't work, since in a higher dimension we will have to deal with nodal hypersurfaces instead of nodal lines. One may eventually think at using an alternative proof, based on a stronger result, if true, which says that the solutions to our problems might have convex level sets. In such a case, G.A. Philippin and A. Safoui have proved in \cite{PS04} that equality sign in (17) can be replaced with the appropriate inequality sign. Unfortunately, as proved by X.-J. Wang in \cite{Wa14}, this convexity result fails to be true in some particular cases, such as $\alpha=0$ in (1) or $\mu =0$ in (3).

2. As for the optimality of our bounds, the equality sign is obtained in our bound estimates from Theorem 1.2 and Theorem 1.3 when the corresponding P-functions are identically constant. However, in Lemmas 2.4 and 3.2 it was already shown that this thing is impossible. Therefore, the bound estimates (9)-(12) are not optimal. As for the bound estimates from Theorems 6.1 and 6.2, the dimension of the space doesn't play a role in our computations, so these results still remain true in higher dimension. Moreover, the equality sign in these estimates holds in the limit as $\Omega$ degenerates into a strip region of width $2d$, while $\mu $ should also be equal to zero in the case of the estimate found in Theorem 6.2.


\end{document}